\documentclass[a4paper,11pt,reqno]{amsart}

\pdfoutput=1

\usepackage[margin=2.35cm, headsep=0.75cm, footskip=0.75cm]{geometry}
\usepackage{amsmath, amsthm, amssymb, amsfonts, mathtools}
\usepackage{color, xcolor}
\usepackage{enumitem}
\usepackage{csquotes}
\usepackage[mathscr]{euscript}
\usepackage[english]{isodate}
\usepackage{pgf, tikz, tikz-cd}
\usepackage[pdfusetitle, colorlinks=true, linkcolor=blue, linkbordercolor=blue, citecolor=red, citebordercolor=red, urlcolor=indigo]{hyperref}
\usepackage[capitalise, nameinlink, noabbrev, nosort]{cleveref}
\usepackage{doi}
\usepackage[english]{babel}

\definecolor{indigo}{HTML}{492DA5}

\usetikzlibrary{shapes,arrows,automata}
\pgfdeclarelayer{bg} 
\pgfsetlayers{bg,main} 

\providecommand{\noopsort}[1]{} 

\setenumerate{listparindent=\parindent}

\makeatletter\g@addto@macro\bfseries{\boldmath}\makeatother

\makeatletter
\let\origsection\section
\renewcommand\section{\@ifstar{\starsection}{\nostarsection}}
\newcommand\sectionspace{\vspace{0.5ex}}
\newcommand\nostarsection[1]{\sectionspace\origsection{#1}\sectionspace}
\newcommand\starsection[1]{\sectionspace\origsection*{#1}\sectionspace}
\makeatother

\setlist[enumerate]{font=\normalfont}
\crefname{enumi}{}{}

\numberwithin{equation}{section}
\crefname{equation}{Equation}{Equations}

\newtheorem{theorem}{Theorem}[section]

\crefname{thm}{Theorem}{Theorems}

\newtheorem{lemma}[theorem]{Lemma}
\crefname{lemma}{Lemma}{Lemmas}

\newtheorem{proposition}[theorem]{Proposition}
\crefname{proposition}{Proposition}{Propositions}

\newtheorem{corollary}[theorem]{Corollary}
\crefname{corollary}{Corollary}{Corollaries}

\theoremstyle{definition}

\newtheorem*{assumption*}{Assumption}
\crefname{assumption}{Assumption}{Assumptions}

\theoremstyle{remark}

\newtheorem{remark}[theorem]{Remark}
\crefname{remark}{Remark}{Remarks}

\newcommand{\E}{E}
\newcommand{\F}{\mathscr{F}}
\renewcommand{\L}{\mathscr{L}}
\newcommand{\T}{\mathscr{T}}
\newcommand{\U}{\mathscr{U}}
\newcommand{\iofname}{\epsilon}
\newcommand{\iof}[1]{\epsilon(#1)}
\newcommand{\iofinv}[1]{\epsilon^{-1}(#1)}
\newcommand{\idempotentfiltersof}[1]{{#1}^{(0)}}

\newcommand{\filterone}{F}
\newcommand{\filtertwo}{G}
\newcommand{\properfilterone}{F}
\newcommand{\properfiltertwo}{G}
\newcommand{\properfilterthree}{H}
\newcommand{\ultrafilterone}{U}

\newcommand{\eF}{{\normalfont\textsf{F}}}
\newcommand{\eT}{{\normalfont\textsf{T}}}
\newcommand{\eU}{{\normalfont\textsf{U}}}
\newcommand{\eG}{{\normalfont\textsf{G}}}
\newcommand{\eL}{{\normalfont\textsf{L}}}
\newcommand{\setoffilters}{\L}
\newcommand{\setofproperfilters}{\F}
\newcommand{\setofultrafilters}{\U}
\newcommand{\ISGoffilters}{\mathscr{L}}
\newcommand{\groupoidoffilters}{\setoffilters}
\newcommand{\groupoidofproperfilters}{\setofproperfilters}
\newcommand{\groupoidoftightfilters}{\setoftightfilters}
\newcommand{\groupoidofultrafilters}{\setofultrafilters}
\newcommand{\groupoidofgerms}{\groupoid}
\newcommand{\groupoidofpropergerms}{\groupoid_0}
\newcommand{\groupoidoftightgerms}{\groupoid_{\text{tight}}}
\newcommand{\groupoidofultragerms}{\groupoid_{\infty}}
\newcommand{\Patersonsgroupoid}{\groupoid_u}
\newcommand{\finsubseteq}{\subseteq_{\normalfont\text{fin}}}
\newcommand{\germequivalence}[3]{\sim_{#1}}
\newcommand{\germ}[2]{[#1, #2]}
\newcommand{\openone}{A}
\newcommand{\npo}{\le}
\newcommand{\nporeverse}{\ge}
\newcommand{\up}[1]{#1^\uparrow}
\newcommand{\upbrackets}[1]{(#1)^\uparrow}
\newcommand{\down}[1]{#1^\downarrow}
\newcommand{\indexbyin}[3]{(#1_{#2})_{#2 \in #3}}
\newcommand{\filtercompose}[2]{#1 \cdot #2}
\newcommand{\sourcesymbolname}{\normalfont\textbf{d}}
\newcommand{\sourcesymbol}[1]{\normalfont\textbf{d}(#1)}
\newcommand{\inversesourcesymbol}[1]{\normalfont\textbf{d}^{-1}(#1)}
\newcommand{\rangesymbolname}{\normalfont\textbf{r}}
\newcommand{\rangesymbol}[1]{\normalfont\textbf{r}(#1)}
\newcommand{\lff}{\lambda}
\newcommand{\pff}{\pi}
\newcommand{\Efilterone}{\xi}
\newcommand{\Efiltertwo}{\eta}
\newcommand{\Efilterofproperfilter}[1]{\Efilterone_{#1}}
\newcommand{\setoftightfilters}{\T}
\newcommand{\setofEfilters}{\eL(\E)}
\newcommand{\setofEproperfilters}{\eF(\E)}
\newcommand{\setoftightEfilters}{\eT(\E)}
\newcommand{\setofEultrafilters}{\eU(\E)}
\newcommand{\standardSEf}[1]{\beta_{#1}}
\newcommand{\gogrepset}[3]{\Omega_{#1}}
\newcommand{\zero}{0}
\newcommand{\groupoid}{\mathcal{G}}
\newcommand{\groupoidH}{\mathcal{H}}
\newcommand{\units}[1]{#1^{(0)}}
\newcommand{\comps}[1]{#1^{(2)}}
\newcommand{\ghomomorphism}{\phi}

\newcommand{\properfiltersIB}[1]{\F_{#1}}
\newcommand{\properfiltersPB}[2]{\F_{#1:#2}}
\newcommand{\properfiltersIBcollection}{(\properfiltersIB{s})_{s \in S}}
\newcommand{\properfiltersPBcollection}{(\properfiltersPB{s}{T})_{s \in S, T \finsubseteq \down{s}}}
\newcommand{\unitsproperfiltersPBcollection}{(\properfiltersPB{e}{X})_{e \in \E, X \finsubseteq \down{e}}}
\newcommand{\ultrafiltersIB}[1]{\U_{#1}}
\newcommand{\ultrafiltersIBcollection}{(\ultrafiltersIB{s})_{s \in S}}
\newcommand{\EfiltersIB}[1]{\eF_{#1}}
\newcommand{\EfiltersPB}[2]{\eF_{#1:#2}}
\newcommand{\EultrafiltersIB}[1]{\eU_{#1}}
\newcommand{\propergermsbasic}[2]{\Theta_0(#1,#2)}
\newcommand{\ultragermsbasic}[2]{\Theta_\infty(#1,#2)}
\newcommand{\inv}[1]{#1^{-1}}
\newcommand{\sinv}[1]{#1^*}
\newcommand{\finv}[1]{\sinv{#1}}
\newcommand{\ginv}[1]{#1^{-1}}

\hypersetup{pdfauthor={Becky Armstrong, Lisa Orloff Clark, Astrid an Huef, Malcolm Jones, and Ying-Fen Lin}, pdftitle={Filtering germs: Groupoids associated to inverse semigroups}}

\title[Filtering germs: Groupoids associated to inverse semigroups]{Filtering germs: \\ Groupoids associated to inverse semigroups}

\author[Armstrong]{Becky Armstrong}
\author[Clark]{Lisa Orloff Clark}
\author[an Huef]{Astrid an Huef}
\author[Jones]{Malcolm Jones}
\author[Lin]{Ying-Fen Lin}

\address[B.\ Armstrong]{Mathematical Institute, WWU M\"unster, Einsteinstr.\ 62, 48149 M\"unster, GERMANY}
\email{\href{mailto:becky.armstrong@uni-muenster.de}{becky.armstrong@uni-muenster.de}}

\address[L.O.\ Clark, A.\ an Huef, and M.\ Jones]{School of Mathematics and Statistics, Victoria University of Wellington, PO Box 600, Wellington 6140, NEW ZEALAND}
\email{\href{mailto:lisa.clark@vuw.ac.nz}{lisa.clark}, \href{mailto:astrid.anhuef@vuw.ac.nz}{astrid.anhuef}, \href{mailto:malcolm.jones@vuw.ac.nz}{malcolm.jones@vuw.ac.nz}}

\address[Y.-F.\ Lin]{Mathematical Sciences Research Centre, Queen's University Belfast, Bel\-fast, BT7 1NN, UNITED KINGDOM}
\email{\href{mailto:y.lin@qub.ac.uk}{y.lin@qub.ac.uk}}

\thanks{This research is supported by the Marsden Fund grant 18-VUW-056 from the Royal Society of New Zealand.}


\date{\today.}

\begin{document}

\begin{abstract}
We investigate various groupoids associated to an arbitrary inverse semigroup with zero. We show that the groupoid of filters with respect to the natural partial order is isomorphic to the groupoid of germs arising from the standard action of the inverse semigroup on the space of idempotent filters. We also investigate the restriction of this isomorphism to the groupoid of tight filters and to the groupoid of ultrafilters.
\end{abstract}

\maketitle

\section{Introduction}

An \emph{inverse semigroup} is a set $S$ endowed with an associative binary operation such that for each $a \in S$, there is a unique $\sinv{a} \in S$, called the \emph{inverse} of $a$, satisfying
\[
a\sinv{a}a = a \ \text{ and } \ \sinv{a}a\sinv{a} = \sinv{a}.
\]
The study of \'etale groupoids associated to inverse semigroups was initiated by Renault \cite[Remark~III.2.4]{Ren80}. We consider two well known groupoid constructions: the filter approach and the germ approach, and we show that the two approaches yield isomorphic groupoids.

Every inverse semigroup has a natural partial order, and a filter is a nonempty down-directed up-set with respect to this order. The filter approach to groupoid construction first appeared in \cite{Len08}, and was later simplified in \cite{LMS13}. Work in this area is ongoing; see for instance, \cite{Bic21, BC20, Cas20}.

Every inverse semigroup acts on the filters of its subsemigroup of idempotents. The groupoid of germs associated to an inverse semigroup encodes this action. Paterson pioneered the germ approach in \cite{Pat99} with the introduction of the universal groupoid of an inverse semigroup. Exel's treatise \cite{Exe08} gives a construction of a groupoid of germs akin to Paterson's, and introduces the tight groupoid. The germ approach is also used in the construction of Kumjian and Renault's Weyl groupoid \cite{Kum86, Ren08}.

The filter and germ approaches yield the same groupoids up to isomorphism. While this is mentioned in the literature (see \cite[Theorem~4.13]{Len08}, \cite[Corollary~5.7]{MR10}, \cite[Section~3.3]{LMS13}, \cite[Section~5.1]{LL13}, \cite[Remark~2.43]{Bic21}, and \cite[Corollary~6.8]{Cas20}), the details are omitted or are deeply embedded in proofs. In our main theorem (\cref{thm:pff}) we give an explicit isomorphism from the groupoid of proper filters to the groupoid of proper germs; we also give a formula for its inverse in terms of the idempotents of the inverse semigroup. The advantage of our approach is that we can then deduce the corresponding results for the groupoids of tight filters and ultrafilters via restriction. This gives a detailed and unified exposition which has already been a useful tool in \cite{ACCCLMRSS21}.

We consider a number of groupoids in this work; our notation for these is as follows:
\begin{itemize}
\item $\groupoidofproperfilters$ is the groupoid of proper filters,
\item $\groupoidoftightfilters$ is the groupoid of tight filters,
\item $\groupoidofultrafilters$ is the groupoid of ultrafilters,
\item $\groupoidofpropergerms$ is the groupoid of proper germs,
\item $\groupoidoftightgerms$ is Exel's tight groupoid, and
\item $\groupoidofultragerms$ is the groupoid of ultragerms.
\end{itemize}
These groupoids are related via the diagram
\begin{center} \phantomsection \label{diag:roadmap}
\begin{tikzcd}[column sep=tiny]
\groupoidofproperfilters \arrow[d, hook, two heads, "\pff"] & \ge & \groupoidoftightfilters \arrow[d, hook, two heads, "\pff|_{\groupoidoftightfilters}"] & \ge & \groupoidofultrafilters \arrow[d, hook, two heads, "\pff|_{\groupoidofultrafilters}"] \\
\groupoidofpropergerms & \ge & \groupoidoftightgerms & \ge & \groupoidofultragerms
\end{tikzcd}
\end{center}
where $\groupoid \ge \groupoidH$ means that $\groupoidH$ is a subgroupoid of $\groupoid$, and the maps are topological groupoid isomorphisms. Note that $\ge$ is an order relation.

The objective of \cref{sec:homeo_unit_spaces} is to show that the unit spaces of $\groupoidofproperfilters$ and $\groupoidofpropergerms$ are homeomorphic. In \cref{sec:preliminaries} we present some background on groupoids and inverse semigroups, and in \cref{sec:filters} we define filters of inverse semigroups. We then introduce the various patch topologies in \cref{sec:patch}, and finally in \cref{sec:unit_spaces} we describe our homeomorphism. In \cref{sec:groupoids} we define the groupoids themselves, and we note key interactions between the structure of filters and germs (see, for instance, \cref{lem:pff_wd_rp}). In \cref{sec:filters_to_germs} we show that $\groupoidofproperfilters$ and $\groupoidofpropergerms$ are isomorphic as topological groupoids (\cref{thm:pff}), and that our isomorphism restricts to an isomorphism from $\groupoidofultrafilters$ to $\groupoidofultragerms$ (\cref{cor:uff}).

\section{Idempotents, filters, and the patch topology} \label{sec:homeo_unit_spaces}

The main goal of this paper is to show that the groupoid of proper filters is isomorphic to the groupoid of proper germs, and that the groupoid of ultrafilters is isomorphic to the groupoid of ultragerms. In this section, we show that the unit spaces of these pairs of groupoids are homeomorphic with respect to the patch topology.

\subsection{Preliminaries} \label{sec:preliminaries}

A \emph{groupoid} is a set $\groupoid$ with a subset $\comps{\groupoid} \subseteq \groupoid \times \groupoid$, a \emph{composition} $(\alpha,\beta) \mapsto \alpha \beta$ from $\comps{\groupoid}$ to $\groupoid$ and an \emph{inversion} $\gamma \mapsto \inv{\gamma}$ on $\groupoid$ such that, for all $\alpha,\beta,\gamma \in \groupoid$,
\begin{enumerate}[label=(\roman*)]
\item $(\gamma^{-1})^{-1}=\gamma$ and $(\inv{\gamma},\gamma) \in \comps{\groupoid}$,
\item if $(\alpha, \beta),(\beta,\gamma) \in \comps{\groupoid}$, then $(\alpha, \beta \gamma),(\alpha \beta,\gamma) \in \comps{\groupoid}$ and $\alpha (\beta \gamma) = (\alpha \beta) \gamma$, and
\item if $(\gamma,\eta) \in \comps{\groupoid}$, then $\gamma^{-1}\gamma\eta = \eta$ and $\gamma\eta\eta^{-1} = \gamma$.
\end{enumerate}
The \emph{units} of a groupoid $\groupoid$ are the elements of the \emph{unit space} $\units{\groupoid} \coloneqq \{\inv{\gamma}\gamma : \gamma \in \groupoid\}$. The \emph{source} map $\sourcesymbolname\colon \groupoid \to \groupoid$ is defined by $\sourcesymbol{\gamma} \coloneqq \inv{\gamma}\gamma$, and the \emph{range} map $\rangesymbolname\colon \groupoid \to \groupoid$ is defined by $\rangesymbol{\gamma} \coloneqq \gamma\inv{\gamma}$. A subset $A$ of $\groupoid$ is called a \emph{local bisection} if $\sourcesymbolname|_A\colon A \to \groupoid$ is injective, or, equivalently, if $\rangesymbolname|_A\colon A \to \groupoid$ is injective. A subset $\groupoidH$ of $\groupoid$ is a \emph{subgroupoid} if, for all $\gamma \in \groupoidH$ and $(\alpha,\beta) \in (\groupoidH \times \groupoidH) \cap \comps{\groupoid}$, we have $\gamma^{-1}, \alpha\beta \in \groupoidH$. Given groupoids $\groupoid$ and $\groupoidH$, a bijection $\ghomomorphism\colon \groupoid \to \groupoidH$ is called a \emph{groupoid isomorphism} if for all $(\alpha,\beta) \in \comps{\groupoid}$, we have $(\ghomomorphism(\alpha),\ghomomorphism(\beta)) \in \comps{\groupoidH}$ and $\ghomomorphism(\alpha\beta) = \ghomomorphism(\alpha)\ghomomorphism(\beta)$. We call a groupoid $\groupoid$ \emph{topological} if $\groupoid$ is endowed with a topology, with respect to which composition and inversion are continuous (and consequently, so are $\sourcesymbolname$ and $\rangesymbolname$). A topological groupoid $\groupoid$ is \emph{\'etale} if its source map (or, equivalently, its range map) is a local homeomorphism. A basis $\mathscr{B}$ for the topology on a topological groupoid $\groupoid$ is called \emph{\'etale} if, for all $O,N \in \mathscr{B}$, we have $\ginv{O}, ON \in \mathscr{B}$ and $\ginv{O}O \subseteq \units{\groupoid}$. The topology on $\groupoid$ has an \'etale basis if and only if the groupoid $\groupoid$ is \'etale \cite[Proposition~6.6]{BS19}. If $\groupoid$ and $\groupoidH$ are topological groupoids and $\ghomomorphism\colon \groupoid \to \groupoidH$ is a homeomorphism and a groupoid isomorphism, then we call $\ghomomorphism$ a \emph{topological groupoid isomorphism}.

\begin{assumption*}
Throughout, let $S$ be an inverse semigroup containing an element $0$ satisfying $0a = a0 = 0$, for all $a \in S$. Note that even if $S$ doesn't contain such an element, a $0$ can always be adjoined.\footnote{Quoting \cite{Exe09}: \textquote{\textellipsis one may wonder why in the world would anyone want to insert a zero in an otherwise well behaved semigroup. Rather than shy away from inverse semigroups with zero, we will assume that all of them contain a zero element, not least because we want to keep a close eye on this exceptional element.}}
\end{assumption*}

The multiplication and inversion operations in $S$ satisfy $\sinv{(\sinv{a})} = a$ and $\sinv{(ab)} = \sinv{b}\sinv{a}$, for all $a,b \in S$. The \emph{natural partial order} on $S$ is the relation $a \npo b \iff a = a\sinv{a}b$, which satisfies
\begin{itemize}
\item $0 \npo a$,
\item $a \npo b \iff \sinv{a} \npo \sinv{b}$,
\item $a \npo b$ and $c \npo d \implies ac \npo bd$, and
\item $\sinv{(ab)}ab \npo \sinv{b}b$,
\end{itemize}
for all $a,b,c,d \in S$.

An element $e \in S$ is called an \emph{idempotent} if $ee = e$. The set $\E$ of idempotents in $S$ satisfies $\E = \{\sinv{s}s : s \in S\}$. For all $e, f \in \E$, the greatest lower bound of $\{e,f\}$ is $ef$. The set $\E$ is a commutative inverse semigroup. For all $e \in \E$ and $a \in S$, we have
\begin{itemize}
\item $0 \in \E$,
\item $\sinv{e} = e$,
\item $\sinv{a}ea \in \E$,
\item $ae,ea \npo a$, and
\item $a \npo e \implies a \in \E$.
\end{itemize}
We refer to \cite{Pet84, Law98} for extensive treatments of inverse semigroups.

\subsection{Filters, idempotent filters, and filters of idempotents} \label{sec:filters}

Given $A \subseteq S$, we write
\[
\up{A} \coloneqq \{b \in S : a \npo b \text{ for some } a \in A\} \quad \text{ and } \quad \down{A} \coloneqq \{b \in S : b \npo a \text{ for some } a \in A\}.
\]
For each $a \in S$, we write $\up{a} \coloneqq \up{\{a\}}$ and $\down{a} \coloneqq \down{\{a\}}$. For all $e \in \E$, we have $\down{e} \subseteq \E$. We say that $A$ is \emph{down-directed} if, for all $a,b \in A$, there is $c \in A$ such that $c \npo a,b$, and we say that $A$ is an \emph{up-set} if $\up{A} = A$. A \emph{filter} in $S$ is a nonempty down-directed up-set. A filter $\properfilterone$ is \emph{proper} if $0 \notin \properfilterone$. A filter $\ultrafilterone$ that is maximal among proper filters is called an \emph{ultrafilter}. Ultrafilters are prevalent in $S$ by a standard appeal to Zorn's lemma. The sets of filters, proper filters, and ultrafilters in $S$ are denoted by $\setoffilters$, $\setofproperfilters$, and $\setofultrafilters$, respectively. A filter $\filterone$ in $S$ is called \emph{idempotent} if $E \cap \filterone \ne \varnothing$. For any subset $\mathscr{H}$ of $\setoffilters$, we denote by $\idempotentfiltersof{\mathscr{H}}$ the set of filters in $\mathscr{H}$ that are idempotent.

Since $\E$ is an inverse semigroup containing $0$, it has its own set $\setofEproperfilters$ of proper filters and set $\setofEultrafilters$ of ultrafilters, which we call \emph{$\E$-filters} and \emph{$\E$-ultrafilters}, respectively.

\subsection{Patch topology} \label{sec:patch}

Given a set $Y$, we write $A \finsubseteq Y$ when $A$ is a finite subset of $Y$. For all $s \in S$ and all $T \finsubseteq \down{s}$, we define
\[
\properfiltersIB{s} \coloneqq \{\properfilterone \in \setofproperfilters : s \in \properfilterone\} \ \text{ and } \ \properfiltersPB{s}{T} \coloneqq \{\properfilterone \in \setofproperfilters : s \in \properfilterone \subseteq S \setminus T\}.
\]
Then $\properfiltersPBcollection$ is a basis for a topology \cite[Proposition~4.1]{Len08}, called the \emph{patch} topology on $\setofproperfilters$. The collection $(\units{\groupoidofproperfilters}\cap\properfiltersPB{s}{T})_{s \in S,T\finsubseteq\down{s}}$ is a basis for the patch topology on $\units{\setofproperfilters}$, but it will be helpful in computations to have a basis indexed only by idempotents.

\begin{lemma} \label{lem:basis_for_idempotent_proper_filters}
The collection $\unitsproperfiltersPBcollection$ is a basis for the patch topology on $\units{\setofproperfilters}$.
\end{lemma}

\begin{proof}
Fix $s \in S$ and let $T \finsubseteq\down{s}$. Suppose that $\properfilterone \in \units{\groupoidofproperfilters}\cap\properfiltersPB{s}{T}$. It suffices to find $e \in \E$ and $X \finsubseteq\down{e}$ such that $\properfilterone \in \properfiltersPB{e}{X} \subseteq \units{\groupoidofproperfilters}\cap\properfiltersPB{s}{T}$. Since $\properfilterone \in \units{\setofproperfilters}$, there exists $e \in \E \cap \properfilterone$. Since $\down{e} \subseteq \E$, we may assume without loss of generality that $e \npo s$. Let $X \coloneqq \{e\sinv{t}t : t \in T\}$.

We show that $\properfilterone \in \properfiltersPB{e}{X}$. Since $e \in \properfilterone$, it remains to show that $\properfilterone \subseteq S \setminus X$. Suppose, looking for a contradiction, that there is $t \in T$ such that $e\sinv{t}t \in \properfilterone$. We know that $e = es$ because $e \npo s$. Also, $s\sinv{t}t = t$ because $t \in \down{s}$. Hence $e\sinv{t}t = (es)\sinv{t}t = et \npo t$, and so $t \in \properfilterone$. However, $\properfilterone \subseteq S \setminus T$, so we have a contradiction. Therefore, $\properfilterone \subseteq S \setminus X$, and so $\properfilterone \in \properfiltersPB{e}{X}$.

Now we show that $\properfiltersPB{e}{X} \subseteq \units{\groupoidofproperfilters}\cap\properfiltersPB{s}{T}$. Fix $\properfiltertwo \in \properfiltersPB{e}{X}$. Since $e \npo s$, we know $s \in \properfiltertwo$, and $\properfiltertwo \in \units{\setofproperfilters}$ because $e$ is an idempotent. It remains to show that $\properfiltertwo \subseteq S \setminus T$. Suppose, looking for a contradiction, that $t \in \properfiltertwo$ for some $t \in T$. Since $\properfiltertwo$ contains an idempotent, we know that $\properfiltertwo$ is closed under inversion and multiplication (see, for example, \cite[Lemma~2.9]{Law10}). Thus $e,t \in \properfiltertwo$ implies that $e\sinv{t}t \in \properfiltertwo$. However, $e\sinv{t}t \in X$ and $\properfiltertwo \subseteq S \setminus X$, which is a contradiction. Hence, $\properfiltertwo \subseteq S \setminus T$, as required. Therefore, $\properfilterone \in \properfiltersPB{e}{X} \subseteq \units{\groupoidofproperfilters}\cap\properfiltersPB{s}{T}$.
\end{proof}

Analogously, for all $e \in \E$ and all $X \finsubseteq \down{e}$, we define
\[
\EfiltersIB{e} \coloneqq \{\Efilterone \in \setofEproperfilters : e \in \Efilterone\} \ \text{ and } \ \EfiltersPB{e}{X} \coloneqq \{\Efilterone \in \setofEproperfilters : e \in \Efilterone \subseteq \E \setminus X\}.
\]
The collection $(\EfiltersPB{e}{X})_{e \in \E, X \finsubseteq \down{e}}$ is a basis for a topology on $\setofEproperfilters$ \cite{Exe08, Law12, EP16}, called the \emph{patch} topology on $\setofEproperfilters$. The closure of $\setofEultrafilters$ in $\setofEproperfilters$ with respect to the patch topology is denoted by $\setoftightEfilters$, and this coincides with the set $\setoftightEfilters$ mentioned in the proof of \cite[Proposition~2.25]{Law12}.

\subsection{Hausdorff unit space} \label{sec:unit_spaces}

In \cref{sec:groupoids} we define the groupoids $\groupoidofproperfilters$ and $\groupoidofpropergerms$ whose unit spaces $\units{\groupoidofproperfilters}$ and $\units{\groupoidofpropergerms}$ can both be identified with $\setofEproperfilters$. First we identify $\units{\groupoidofproperfilters}$ with $\setofEproperfilters$ using a generalisation of the map given in \cite[Lemma~2.18]{Law12}.

\begin{proposition} \label{prop:gopfunits_homeomorphic_to_setofefilters_patch}
There is a bijection $\iofname\colon \idempotentfiltersof{\setofproperfilters} \to \setofEproperfilters$ given by $\iof{\properfilterone} = \properfilterone \cap \E$, with inverse given by $\iofinv{\Efilterone} = \up{\Efilterone}$. Moreover, $\iofname$ is a homeomorphism of Hausdorff spaces with respect to the patch topology, and the restriction $\iofname|_{\idempotentfiltersof{\setofultrafilters}}$ is a homeomorphism onto $\setofEultrafilters$.
\end{proposition}

\begin{proof}
To see that $\iofname$ is well defined, note that for each $\properfilterone \in \units{\groupoidofproperfilters}$, we have $\properfilterone \cap \E \in \setofEproperfilters$, because $\down{e} \subseteq \E$ for each $e \in \E$. Note also that for each $\Efilterone \in \setofEproperfilters$, we have $\up{\Efilterone} \in \units{\groupoidofproperfilters}$. To see that $\iofname$ is bijective with inverse $\Efilterone \mapsto \up{\Efilterone}$, it suffices to show that $\properfilterone = \up{\iof{\properfilterone}}$ and $\Efilterone = \iof{\up{\Efilterone}}$, for all $\properfilterone \in \units{\groupoidofproperfilters}$ and $\Efilterone \in \setofEproperfilters$. Fix $\properfilterone \in \units{\groupoidofproperfilters}$. Toward $\properfilterone \subseteq \up{\iof{\properfilterone}}$, we take $x \in \properfilterone$. Since $\properfilterone \in \units{\groupoidofproperfilters}$, there is some $e \in \iof{\properfilterone}$. Since $\properfilterone$ is down-directed, there exists $f \in \properfilterone$ such that $f \npo e,x$. Now $f \in \down{\E} = \E$, so $f \in \iof{\properfilterone}$, and thus $x \in \up{\iof{\properfilterone}}$. For the reverse containment, let $x \in \up{\iof{\properfilterone}}$, and choose $e \in \properfilterone \cap \E$ such that $e \npo x$. It follows that $x \in \properfilterone$, and thus $\properfilterone = \up{\iof{\properfilterone}}$. Now fix $\Efilterone \in \setofEproperfilters$. Since $\Efilterone \in \setofEproperfilters$ is an up-set in $\E$, we have $\iof{\up{\Efilterone}} = \up{\Efilterone} \cap \E = \Efilterone \cap \E = \Efilterone$, as required.

To see that $\iofname$ is a homeomorphism, observe that $\iof{\properfiltersPB{e}{X}} = \EfiltersPB{e}{X}$, for all $e \in \E$ and all $X \finsubseteq\down{e}$, so \cref{lem:basis_for_idempotent_proper_filters} implies the result. Note that $\setofEproperfilters$ is Hausdorff with respect to the patch topology by \cite[Lemma~2.22]{Law12} (which applies because $\E$ is a meet-semilattice). Since $\epsilon$ is a homeomorphism, $\idempotentfiltersof{\setofproperfilters}$ is Hausdorff as well.

To see that $\iofname|_{\idempotentfiltersof{\setofultrafilters}}$ is a bijection from $\idempotentfiltersof{\setofultrafilters}$ to $\setofEultrafilters$, it suffices to show that $\iof{\units{\setofultrafilters}} = \setofEultrafilters$, because $\iofname\colon \idempotentfiltersof{\setofproperfilters} \to \setofEproperfilters$ is a bijection. Fix $\ultrafilterone \in \units{\setofultrafilters}$. Suppose that $\iof{\ultrafilterone} \subseteq \Efilterone \in \setofEproperfilters$. Then $\ultrafilterone \subseteq \iofinv{\Efilterone}$, but $\ultrafilterone$ is an ultrafilter, so $\ultrafilterone = \iofinv{\Efilterone}$. Hence $\iof{\ultrafilterone} = \Efilterone$, and so $\iof{\ultrafilterone} \in \setofEultrafilters$. A similar argument yields the reverse inclusion, so $\iof{\units{\setofultrafilters}} = \setofEultrafilters$.
\end{proof}

\section{Groupoids associated to inverse semigroups} \label{sec:groupoids}

Given a map $x \mapsto \sinv{x}$ on a set $X$, we define $\finv{A} \coloneqq \{\sinv{a} : a \in A\}$, for all $A \subseteq X$. Given $\comps{X} \subseteq X \times X$ and a map $(x,y) \mapsto xy$ from $\comps{X}$ to $X$, we define $AB \coloneqq \{ab : (a,b) \in (A \times B) \cap \comps{X}\}$, $xB \coloneqq \{x\}B$, and $Ay \coloneqq A\{y\}$, for all $A, B \subseteq X$ and $x, y \in X$.

\subsection{Groupoids of filters}

For all $\filterone,\filtertwo \in \setoffilters$, we define $\filtercompose{\filterone}{\filtertwo} \coloneqq \upbrackets{\filterone\filtertwo} \subseteq S$. Then $\ISGoffilters$ is an inverse semigroup with multiplication given by $(\filterone,\filtertwo) \mapsto \filtercompose{\filterone}{\filtertwo}$ and inversion given by $\filterone \mapsto \finv{\filterone}$. By \cite[Theorem~3.8]{LMS13}, $\ISGoffilters$ is isomorphic to Lenz's semigroup $\mathscr{O}(S)$ defined in \cite[Theorem~3.1]{Len08}. By restricting multiplication in $\ISGoffilters$ to $\comps{\groupoidoffilters} \coloneqq \{(\filterone,\filtertwo) \in \groupoidoffilters \times \groupoidoffilters : \filtercompose{\finv{\filterone}}{\filterone} = \filtercompose{\filtertwo}{\finv{\filtertwo}}\}$, we obtain a groupoid (see \cite[Proposition~4 of Section~3.1]{Law98}). Under these operations, the set $\groupoidofproperfilters$ forms a subgroupoid of $\groupoidoffilters$, which we call the \emph{groupoid of proper filters}. (This is the groupoid denoted by $\eL(S)$ in \cite{LL13}.) It is a consequence of \cite[Proposition~1 of Section~9.2]{Law98} that $\groupoidofultrafilters$ is a subgroupoid of $\groupoidofproperfilters$, which we call the \emph{groupoid of ultrafilters}. The groupoid of ultrafilters is the Weyl groupoid from \cite[Definition~2.42]{Bic21}, which is denoted by $\eG(S)$ in \cite{Law12, LL13}. The unit space of $\groupoidofproperfilters$ is precisely the set $\units{\groupoidofproperfilters}$ of idempotent proper filters, and the unit space of $\groupoidofultrafilters$ is $\units{\groupoidofultrafilters}$.

The following lemma is used to prove \cref{lem:pff_onto_lemma}. A more general version can be found in \cite[Proposition~1.4(ii)]{Law93} (see also \cite[Lemmas~2.8~and~2.11]{Law10}).

\begin{lemma} \label{lem:calc}
Fix $\properfilterone, \properfiltertwo \in \groupoidofproperfilters$.
\begin{enumerate}[label=(\alph*)]
\item \label{item1:lemcalc} For all $s \in \properfilterone$, we have $\properfilterone = \upbrackets{s\sourcesymbol{\properfilterone}}$.
\item \label{item2:lemcalc} If $\properfilterone \cap \properfiltertwo \ne \varnothing$ and $\sourcesymbol\properfilterone = \sourcesymbol\properfiltertwo$, then $\properfilterone=\properfiltertwo$.
\end{enumerate}
\end{lemma}

\begin{proof}
For part~\cref{item1:lemcalc}, fix $s, x \in \properfilterone$. Then there exists $y \in \properfilterone$ such that $y \npo x, s$. Thus $s\sinv{y}y = y \npo x$, and so $x \in \upbrackets{s\sourcesymbol{\properfilterone}}$. Since $s\sourcesymbol{\properfilterone} \subseteq \properfilterone$ and $\properfilterone$ is an up-set, the reverse containment is clear.

For part~\cref{item2:lemcalc}, suppose that $s \in \properfilterone \cap \properfiltertwo$ and $\sourcesymbol\properfilterone = \sourcesymbol\properfiltertwo$. Then part~\cref{item1:lemcalc} gives
\[
\properfilterone = \upbrackets{s\sourcesymbol{\properfilterone}} = \upbrackets{s\sourcesymbol{\properfiltertwo}} = \properfiltertwo. \qedhere
\]
\end{proof}

The groupoid $\groupoidofproperfilters$ of proper filters has Hausdorff unit space $\units{\groupoidofproperfilters}$ with respect to the patch topology, by \cref{prop:gopfunits_homeomorphic_to_setofefilters_patch}. We will show that $\groupoidofproperfilters$ is an \'etale groupoid with respect to the patch topology, but to see this, we will first consider a coarser topology on $\groupoidofproperfilters$ generated by the subcollection $\properfiltersIBcollection$, which is a basis by \cite{Bic21, Cas20}. Interestingly, in the subspace $\setofultrafilters$ of ultrafilters, the topology generated by the basis $(\properfiltersIB{s} \cap \setofultrafilters)_{s \in S}$ is the entire patch topology (see \cref{prop:ultrafilters_inclusive}).

\begin{lemma} \label{lem:pff_onto_lemma}
For all $s, t \in S$, we have 
\begin{enumerate}[label=(\alph*)]
\item \label{closed_under_inverse} $\sinv{\properfiltersIB{s}} = \properfiltersIB{\sinv{s}}$,
\item \label{fsft_fst} $\properfiltersIB{s}\properfiltersIB{t} = \properfiltersIB{st}$,
\item \label{local_bisection} $\properfiltersIB{s}$ is a local bisection in $\groupoidofproperfilters$,
\item \label{dfs_fsinvss} $\sourcesymbol{\properfiltersIB{s}} = \properfiltersIB{\sinv{s}s} ~\subseteq \units{\groupoidofproperfilters}$, and
\item \label{etale_basis} $\properfiltersIBcollection$ is an \'etale basis for the topology on the groupoid $\groupoidofproperfilters$.
\end{enumerate}
\end{lemma}

\begin{proof}
Fix $s,t \in S$. Part~\cref{closed_under_inverse} follows from the definitions and that $\sinv{(\sinv{s})} = s$. Our argument for part~\cref{fsft_fst} is inspired by the proof of \cite[Lemma~2.10(4)]{Law12}. It is clear that $\properfiltersIB{s}\properfiltersIB{t} \subseteq \properfiltersIB{st}$. Fix $\properfilterthree \in \properfiltersIB{st}$. Put $\properfilterone \coloneqq \up{(s\up{(t \sourcesymbol{\properfilterthree} \sinv{t})})}$ and $\properfiltertwo \coloneqq \up{(t\sourcesymbol{\properfilterthree})}$. The subsets $\properfilterone$ and $\properfiltertwo$ are filters because they are upward closures of down-directed sets.

Since $\sinv{s}s \in E$, we have
\[
\sinv{t}t \nporeverse \sinv{t}(\sinv{s}s)t =~ \sinv{(st)}st \in \sourcesymbol{\properfilterthree},
\]
and so $\sinv{t}t \in \sourcesymbol{\properfilterthree}$. Thus, $t = t\sinv{t}t \in \properfiltertwo$, so $\properfiltertwo \in \properfiltersIB{t}$. Since $t\sinv{t} \in \E$, we have
\[
\sinv{s}s \nporeverse \sinv{s}st\sinv{t} \nporeverse t\sinv{t}\sinv{s}st\sinv{t} = t \sinv{(st)}st \sinv{t} \in t \sourcesymbol{\properfilterthree} \sinv{t},
\]
and so $\sinv{s}s \in \up{(t \sourcesymbol{\properfilterthree} \sinv{t})}$. Thus, $s = s\sinv{s}s \in \properfilterone$, so $\properfilterone \in \properfiltersIB{s}$.

We now show that $\sourcesymbol{\properfilterone} \subseteq \rangesymbol{\properfiltertwo}$. Fix $x \in \properfilterone$. We must show that $\sinv{x}x \in \rangesymbol{\properfiltertwo}$. Choose $h \in \properfilterthree$ such that $st\sinv{h}h\sinv{t} \npo x$. Since $h, st \in \properfilterthree$, we may assume without loss of generality that $h \npo st$, and so $x \nporeverse st\sinv{h}h\sinv{t} ~\nporeverse~ h\sinv{t}$. Thus,
\[
\sinv{x}x \nporeverse \sinv{(h\sinv{t})}h\sinv{t} = t\sinv{h}h\sinv{t} = t\sinv{h}h\sinv{h}h\sinv{t} = t\sinv{h}h\sinv{(t\sinv{h}h)} \in \rangesymbol{\properfiltertwo},
\]
and so $\sinv{x}x \in \rangesymbol{\properfiltertwo}$. Now we show that $\rangesymbol{\properfiltertwo} \subseteq \sourcesymbol{\properfilterone}$. Fix $y \in \properfiltertwo$. We must show that $y\sinv{y} \in \sourcesymbol{\properfilterone}$. Choose $h \in \properfilterthree$ such that $y \nporeverse t\sinv{h}h$. Since $st\sinv{h}h\sinv{t} \in \properfilterone$ and $t\sinv{h}ht, \sinv{s}s \in \E$, we have
\[
y\sinv{y} \nporeverse t\sinv{h}h\sinv{(t\sinv{h}h)} = t\sinv{h}h\sinv{t} \nporeverse t\sinv{h}h\sinv{t}\sinv{s}st\sinv{h}h\sinv{t} = \sinv{(st\sinv{h}h\sinv{t})}st\sinv{h}h\sinv{t} \in \sourcesymbol{\properfilterone},
\]
and so $y\sinv{y} \in \sourcesymbol{\properfilterone}$. Therefore, $\sourcesymbol{\properfilterone} = \rangesymbol{\properfiltertwo}$.

To see that $\filtercompose{\properfilterone}{\properfiltertwo} = \properfilterthree$, fix $x \in \filtercompose{\properfilterone}{\properfiltertwo}$. Then there exist $h_1,h_2 \in \properfilterthree$ such that
\[
x \nporeverse (st\sinv{h_1}h_1\sinv{t})(t\sinv{h_2}h_2).
\]
Since $h_1,h_2,st \in \properfilterthree$, there exists $h \in \properfilterthree$ such that $h \npo h_1, h_2, st$. Now, since $\sinv{h_1}h_1, \sinv{t}t \in \E$, we have
\[
x \nporeverse st\sinv{h_1}h_1\sinv{t}t\sinv{h_2}h_2 = st\sinv{t}t\sinv{h_1}h_1\sinv{h_2}h_2 = st\sinv{h_1}h_1\sinv{h_2}h_2 \nporeverse st\sinv{h}h \nporeverse h\sinv{h}h = h \in \properfilterthree,
\]
and hence $x \in \properfilterthree$. For the reverse inclusion, fix $h \in \properfilterthree$. We may assume that $h \npo st$, so a similar computation yields $h = (st\sinv{h}h\sinv{t})(t\sinv{h}h) \in \filtercompose{\properfilterone}{\properfiltertwo}$. Hence, $\filtercompose{\properfilterone}{\properfiltertwo} = \properfilterthree$.

It remains to show that $\properfilterone$ and $\properfiltertwo$ are proper filters, so that $\properfilterthree = \filtercompose{\properfilterone}{\properfiltertwo} \in \properfiltersIB{s}\properfiltersIB{t}$. If either $\properfilterone$ or $\properfiltertwo$ contains $\zero$, then so must $\properfilterthree$, because $\properfilterthree = \filtercompose{\properfilterone}{\properfiltertwo}$. But $\properfilterthree$ is proper, so neither $\properfilterone$ nor $\properfiltertwo$ contains $\zero$, completing the proof.

Part~\cref{local_bisection} follows from \cref{lem:calc}\cref{item2:lemcalc}. For part~\cref{dfs_fsinvss}, note that $\sinv{s}s \in \E$, so $\properfiltersIB{\sinv{s}s} \subseteq \units{\groupoidofproperfilters}$. The inclusion $\sourcesymbol{\properfiltersIB{s}} \subseteq \properfiltersIB{\sinv{s}s}$ follows from parts \cref{closed_under_inverse,fsft_fst}. Part~\cref{fsft_fst} implies that for the reverse inclusion, it is enough to show that $\properfiltersIB{\sinv{s}}\properfiltersIB{s} \subseteq \sourcesymbol{\properfiltersIB{s}}$. Fix $(\properfilterone,\properfiltertwo) \in (\properfiltersIB{\sinv{s}} \times \properfiltersIB{s}) \cap \comps{\groupoidofproperfilters}$, so that $\filtercompose{\properfilterone}{\properfiltertwo} \in \properfiltersIB{\sinv{s}}\properfiltersIB{s}$. Then $\sourcesymbol{\properfilterone} = \sourcesymbol{\finv{\properfiltertwo}}$ and $\properfilterone,\finv{\properfiltertwo} \in \properfiltersIB{\sinv{s}}$, so $\properfilterone = \finv{\properfiltertwo}$ by \cref{lem:calc}\cref{item2:lemcalc}. Thus, $\filtercompose{\properfilterone}{\properfiltertwo} = \filtercompose{\finv{\properfiltertwo}}{\properfiltertwo} = \sourcesymbol{\properfiltertwo} \in \sourcesymbol{\properfiltersIB{s}}$. Part~\cref{etale_basis} is a consequence of parts \cref{closed_under_inverse,fsft_fst,dfs_fsinvss}.
\end{proof}

By \cite[Proposition~6.6]{BS19}, a groupoid endowed with a topology is \'etale if and only if the topology has an \'etale basis. Thus, \cref{prop:gopf_etale_inclusive} is immediate from \cref{lem:pff_onto_lemma}\cref{etale_basis}.

\begin{proposition} \label{prop:gopf_etale_inclusive}
The groupoid $\groupoidofproperfilters$ of proper filters is \'etale with respect to the topology generated by $\properfiltersIBcollection$.
\end{proposition}

\begin{remark}
Although the topology generated by $\properfiltersIBcollection$ is nice enough to make $\groupoidofproperfilters$ \'etale, it can be shown that $\units{\groupoidofproperfilters}$ is non-Hausdorff, provided there are distinct $e,f \in \E$ such that $ef \ne 0$. If the intention is to find a groupoid with a Hausdorff unit space, \cref{prop:ultrafilters_inclusive} tells us that it suffices to restrict to the groupoid $\groupoidofultrafilters$ of ultrafilters. We would rather keep all proper filters on hand for now, so instead of taking a smaller groupoid we refine the topology to the patch topology on the same groupoid.
\end{remark}

\begin{corollary} \label{cor:gopf_etale_Hausdorffunits_patch}
The groupoid $\groupoidofproperfilters$ of proper filters has a Hausdorff unit space and is \'etale with respect to the patch topology.
\end{corollary}

\begin{proof}
The groupoid $\groupoidofproperfilters$ of proper filters is a topological groupoid with respect to the patch topology by \cite[Proposition~4.3]{Len08}. Since the patch topology is finer than the topology generated by $\properfiltersIBcollection$, $\groupoidofproperfilters$ is \'etale by \cref{prop:gopf_etale_inclusive}, and $\units{\groupoidofproperfilters}$ is Hausdorff by \cref{prop:gopfunits_homeomorphic_to_setofefilters_patch}.
\end{proof}

Henceforth, the groupoid $\groupoidofproperfilters$ of proper filters is endowed with the patch topology.

\subsection{Groupoids of germs}

The initial data used to construct a groupoid of germs employs the notion of an `action' of $S$ on a nice topological space $X$ \cite[Definition~3.1]{EP16}. We are exclusively interested in the `standard' action $\beta$ of $S$ on $\setofEproperfilters$ described in \cite[Section~3]{EP16}, so we neglect the details of inverse semigroup actions for brevity, introducing structure only as necessary.

Define
$\gogrepset{}{}{} \coloneqq \{(s,\Efilterone) \in S \times \setofEproperfilters : \Efilterone \in \EfiltersIB{\sinv{s}s}\}$, and define a relation $\germequivalence{}{}{}$ on $\gogrepset{}{}{}$ by $(s,\Efilterone) \germequivalence{}{}{} (t,\Efiltertwo)$ if and only if there is some $e \in \Efilterone = \Efiltertwo$ such that $se = te$. The relation $\germequivalence{}{}{}$ is an equivalence relation on $\gogrepset{}{}{}$. We denote the equivalence class of $(s,\Efilterone)$ by $\germ{s}{\Efilterone}$. Define $\groupoidofpropergerms \coloneqq \gogrepset{}{}{} /\! \germequivalence{}{}{}$. For each $s \in S$ and $\Efilterone \in \EfiltersIB{\sinv{s}s}$, define
\[
\standardSEf{s}(\Efilterone) \coloneqq \{ f \in \E : se\sinv{s} \npo f \text{ for some } e \in \Efilterone \},
\]
as per \cite[Equation~(3.4)]{EP16}. Notice that $\standardSEf{s}(\Efilterone) \in \EfiltersIB{s\sinv{s}}$. Define
\[
\comps{\groupoidofpropergerms} \coloneqq \{(\germ{s}{\Efilterone}, \germ{t}{\Efiltertwo}) \in \groupoidofpropergerms \times \groupoidofpropergerms : \Efilterone = \standardSEf{t}(\Efiltertwo) \} \subseteq \groupoidofpropergerms \times \groupoidofpropergerms.
\]

The set $\groupoidofpropergerms$ with the composable pairs $\comps{\groupoidofpropergerms}$ and the operations
\[
(\germ{s}{\standardSEf{t}(\Efiltertwo)}, \germ{t}{\Efiltertwo}) \mapsto \germ{s}{\standardSEf{t}(\Efiltertwo)}\germ{t}{\Efiltertwo} \coloneqq \germ{st}{\Efiltertwo} \ \text{ and } \ \germ{s}{\Efilterone} \mapsto \ginv{\germ{s}{\Efilterone}} \coloneqq \germ{\sinv{s}}{\standardSEf{s}(\Efilterone)}
\]
forms a groupoid, which we call the \emph{groupoid of proper germs}. The unit space of $\groupoidofpropergerms$ is the set $\units{\groupoidofpropergerms} = \{ \germ{e}{\Efilterone} : \Efilterone \in \EfiltersIB{e} \}$, which may be identified with $\setofEproperfilters$ via the bijection $\germ{e}{\Efilterone} \mapsto \Efilterone$ from \cite[Equation~(3.9)]{EP16}.

Due to \cite[Proposition~12.8]{Exe08} and \cite[Proposition~3.5]{EP16}, we have that
\[
\standardSEf{s}(\EfiltersIB{\sinv{s}s}\cap\setoftightEfilters) \subseteq \setoftightEfilters \ \text{ and } \ \standardSEf{s}(\EfiltersIB{\sinv{s}s}\cap \setofEultrafilters) \subseteq \setofEultrafilters, \ \text{ for all } s \in S.
\]
Thus, $\setoftightEfilters$ and $\setofEultrafilters$ are invariant subsets of $\units{\groupoidofpropergerms}$, and the corresponding reductions $\groupoidoftightgerms$ and $\groupoidofultragerms$ are subgroupoids of $\groupoidofpropergerms$ with unit spaces $\setoftightEfilters$ and $\setofEultrafilters$, respectively. Note that $\groupoidoftightgerms$ is Exel's tight groupoid from \cite[Theorem~13.3]{Exe08}. We call $\groupoidofultragerms$ the \emph{groupoid of ultragerms}.

For each $s \in S$ and each open subset $\openone\subseteq\EfiltersIB{\sinv{s}s}$, define $\propergermsbasic{s}{\openone} \coloneqq \{\germ{s}{\Efilterone} \in \groupoidofpropergerms : \Efilterone \in \openone\}$. The collection of all such sets $\propergermsbasic{s}{\openone}$ is a basis for a topology on $\groupoidofpropergerms$, with respect to which $\groupoidofpropergerms$ is a locally compact \'etale groupoid with a Hausdorff unit space \cite[Section~3]{EP16}.

\begin{lemma} \label{lem:pff_wd_rp}
Fix $\properfilterone \in \groupoidofproperfilters$, and define $\Efilterofproperfilter{\properfilterone} \coloneqq \sourcesymbol{\properfilterone} \cap \E$.
\begin{enumerate}[label=(\alph*)]
\item \label{item:pff_notation} We have $\Efilterofproperfilter{\properfilterone} \in \setofEproperfilters$ and $\Efilterofproperfilter{\ultrafilterone} \in \setofEultrafilters$, for all $\ultrafilterone \in \groupoidofultrafilters$.
\item \label{item:pff_well_defined} For all $s, t \in \properfilterone$, we have $\germ{s}{\Efilterofproperfilter{\properfilterone}} = \germ{t}{\Efilterofproperfilter{\properfilterone}}$.
\item \label{item:pff_range_preserving} For all $s \in \properfilterone$, we have $\standardSEf{s}(\Efilterofproperfilter{\properfilterone}) = \Efilterofproperfilter{\finv{\properfilterone}}$.
\end{enumerate}
\end{lemma}

\begin{proof}
For part~\cref{item:pff_notation}, fix $\properfilterone \in \groupoidofproperfilters$ and $\ultrafilterone \in \groupoidofultrafilters$. Then $\sourcesymbol{\properfilterone} \in \units{\groupoidofproperfilters}$ and $\sourcesymbol{\ultrafilterone} \in \units{\groupoidofultrafilters}$, so \cref{prop:gopfunits_homeomorphic_to_setofefilters_patch} implies that $\Efilterofproperfilter{\properfilterone} = \iof{\sourcesymbol{\properfilterone}} \in \setofEproperfilters$ and $\Efilterofproperfilter{\ultrafilterone} = \iof{\sourcesymbol{\ultrafilterone}} \in \setofEultrafilters$.

For part~\cref{item:pff_well_defined}, fix $s, t \in \properfilterone$. Choose $u \in \properfilterone$ such that $u \npo s,t$. Observe that $\sinv{u}u \in \Efilterofproperfilter{\properfilterone}$, so it suffices to show that $s\sinv{u}u = t\sinv{u}u$, which holds since $s\sinv{u}u = u = t\sinv{u}u$ by the definition of $\npo$.

For part~\cref{item:pff_range_preserving}, fix $s \in \properfilterone$. Since $\Efilterofproperfilter{\finv{\properfilterone}} = \sourcesymbol{\finv{\properfilterone}} \cap \E = \rangesymbol{\properfilterone} \cap \E$, it suffices to show that $\standardSEf{s}(\Efilterofproperfilter{\properfilterone}) = \rangesymbol{\properfilterone} \cap \E$. Fix $f \in \standardSEf{s}(\Efilterofproperfilter{\properfilterone}) \subseteq \E$. Then $se\sinv{s} \npo f$ for some $e \in \Efilterofproperfilter{\properfilterone}$. Since $\sinv{s} \in \finv{\properfilterone}$ and $e \in \Efilterofproperfilter{\properfilterone} = \sourcesymbol{\properfilterone} \cap \E$, we have
\[
f \nporeverse se\sinv{s} \in \filtercompose{\filtercompose{\properfilterone}{\sourcesymbol{\properfilterone}}}{\finv{\properfilterone}} = \filtercompose{\properfilterone}{\finv{\properfilterone}} = \rangesymbol{\properfilterone},
\]
and hence $f \in \rangesymbol{\properfilterone} \cap \E$. It remains to show that $\rangesymbol{\properfilterone} \cap \E \subseteq \standardSEf{s}(\Efilterofproperfilter{\properfilterone})$. Fix $x \in \rangesymbol{\properfilterone} \cap \E$. Since $\rangesymbol{\properfilterone} = \up{(\properfilterone \sinv{\properfilterone})}$, there exists $a \in \properfilterone$ such that $a\sinv{a} \npo x$. Choose $b \in \properfilterone$ such that $b \npo a,s$. Observe that $\sinv{s}a\sinv{a}s \nporeverse \sinv{b}b\sinv{b}b = \sinv{b}b \in \sourcesymbol{\properfilterone}$, and so $\sinv{s}a\sinv{a}s = \sinv{(\sinv{a}s)}\sinv{a}s \in \sourcesymbol{\properfilterone} \cap \E = \Efilterofproperfilter{\properfilterone}$. Now notice that
\[
s(\sinv{s}a\sinv{a}s)\sinv{s} = s\sinv{s}(a\sinv{a})(s\sinv{s}) = s\sinv{s}s\sinv{s}a\sinv{a} = s\sinv{s}a\sinv{a} \npo a\sinv{a} ~\npo x,
\]
from which it follows that $x \in\standardSEf{s}(\Efilterofproperfilter{\properfilterone})$, by the definition of $\standardSEf{s}$.
\end{proof}

\section{Mapping filters to germs} \label{sec:filters_to_germs}

Recall from \cref{lem:pff_wd_rp} that for each $\properfilterone \in \groupoidofproperfilters$, we have $\Efilterofproperfilter{\properfilterone} \coloneqq \sourcesymbol{\properfilterone} \cap \E \in \setofEproperfilters$ and $\germ{s}{\Efilterofproperfilter{\properfilterone}} = \germ{t}{\Efilterofproperfilter{\properfilterone}}$ for all $s, t \in \properfilterone$.

\begin{theorem} \label{thm:pff}
The map $\pff\colon \groupoidofproperfilters \to \groupoidofpropergerms$ given by $\pff(\properfilterone) \coloneqq \germ{s}{\Efilterofproperfilter{\properfilterone}}$ for any $s \in \properfilterone$ is a topological groupoid isomorphism, with inverse given by $\inv{\pff}(\germ{s}{\Efilterone}) = \upbrackets{s\Efilterone}$.
\end{theorem}

\begin{proof}
We begin by showing that $\pff$ is injective. Suppose that $\pff(\properfilterone) = \pff(\properfiltertwo)$ for some $\properfilterone,\properfiltertwo \in\groupoidofproperfilters$. Choose any $s \in \properfilterone$ and $t \in \properfiltertwo$. Then $\germ{s}{\Efilterofproperfilter{\properfilterone}} = \pff(\properfilterone) = \pff(\properfiltertwo) = \germ{t}{\Efilterofproperfilter{\properfiltertwo}}$, so there exists $e \in \Efilterofproperfilter{\properfilterone} = \Efilterofproperfilter{\properfiltertwo}$ such that $se = te$. Since $e \in \Efilterofproperfilter{\properfilterone} \subseteq \sourcesymbol{\properfilterone}$, we have $se \in \filtercompose{\properfilterone}{\sourcesymbol{\properfilterone}} = \properfilterone$; similarly, since $e \in \Efilterofproperfilter{\properfiltertwo} \subseteq \sourcesymbol{\properfiltertwo}$, we have $se = te \in \filtercompose{\properfiltertwo}{\sourcesymbol{\properfiltertwo}} = \properfiltertwo$. Thus, $F, G \in \properfiltersIB{se}$. Recall from \cref{prop:gopfunits_homeomorphic_to_setofefilters_patch} that the map $\iofname\colon \properfilterthree \mapsto \properfilterthree \cap \E$ is a bijection from $\idempotentfiltersof{\setofproperfilters}$ to $\setofEproperfilters$, with inverse given by $\iofinv{\Efilterone} = \up{\Efilterone}$. Thus
\[
\Efilterofproperfilter{\properfilterone} = \Efilterofproperfilter{\properfiltertwo} \implies \iof{\sourcesymbol{\properfilterone}} = \iof{\sourcesymbol{\properfiltertwo}} \implies \sourcesymbol{\properfilterone} = \sourcesymbol{\properfiltertwo}.
\]
Since $\properfiltersIB{se}$ is a local bisection by \cref{lem:pff_onto_lemma}\cref{local_bisection}, we deduce that $\properfilterone = \properfiltertwo$.

Next we show that $\pff$ is surjective. Fix $\germ{s}{\Efilterone} \in \groupoidofpropergerms$. Since $\Efilterone \in \EfiltersIB{\sinv{s}s}$, we have $\iofinv{\Efilterone} \in \properfiltersIB{\sinv{s}s} = \sourcesymbol{\properfiltersIB{s}}$ by \cref{prop:gopfunits_homeomorphic_to_setofefilters_patch} and \cref{lem:pff_onto_lemma}\cref{dfs_fsinvss}. Take $\properfilterone \in \properfiltersIB{s}$ such that $\sourcesymbol{\properfilterone} = \iofinv{\Efilterone}$. Then $s \in \properfilterone$ and $\Efilterofproperfilter{\properfilterone} = \iof{\sourcesymbol{\properfilterone}} = \iof{\iofinv{\Efilterone}} = \Efilterone$. It follows that $\pff(\properfilterone) = \germ{s}{\Efilterone}$, and hence $\pff$ is a bijection. Moreover, \cref{lem:calc}\cref{item1:lemcalc} gives
\[
\inv{\pff}(\germ{s}{\Efilterone}) = \inv{\pff}(\pff(\properfilterone)) = \properfilterone = \upbrackets{s\sourcesymbol{\properfilterone}} = \upbrackets{s \iofinv{\Efilterone}} = \upbrackets{s\up{\Efilterone}} = \upbrackets{s\Efilterone}.
\]

Now we show that $\pff$ is a groupoid isomorphism. Fix $(\properfilterone, \properfiltertwo) \in \comps{\groupoidofproperfilters}$. We must show that $(\pff(\properfilterone), \pff(\properfiltertwo)) \in \comps{\groupoidofpropergerms}$ and $\pff(\filtercompose{\properfilterone}{\properfiltertwo}) = \pff(\properfilterone)\pff(\properfiltertwo)$. Let $s \in \properfilterone$ and $t \in \properfiltertwo$. Then $\pff(\properfilterone) = \germ{s}{\Efilterofproperfilter{\properfilterone}}$ and $\pff(\properfiltertwo) = \germ{t}{\Efilterofproperfilter{\properfiltertwo}}$. By the definition of $\comps{\groupoidofpropergerms}$, we need $\Efilterofproperfilter{\properfilterone} = \standardSEf{t}(\Efilterofproperfilter{\properfiltertwo})$. Since $(\properfilterone, \properfiltertwo) \in \comps{\groupoidofproperfilters}$, we have $\sourcesymbol{\finv{\properfiltertwo}} = \rangesymbol{\properfiltertwo} = \sourcesymbol{\properfilterone}$, and hence \cref{lem:pff_wd_rp}\cref{item:pff_range_preserving} implies that
\[
\standardSEf{t}(\Efilterofproperfilter{\properfiltertwo}) = \Efilterofproperfilter{\finv{\properfiltertwo}} = \iof{\sourcesymbol{\finv{\properfiltertwo}}} = \iof{\sourcesymbol{\properfilterone}} = \Efilterofproperfilter{\properfilterone}.
\]
Now, since $st \in \filtercompose{\properfilterone}{\properfiltertwo}$ and $\Efilterofproperfilter{\filtercompose{\properfilterone}{\properfiltertwo}} = \sourcesymbol{\filtercompose{\properfilterone}{\properfiltertwo}} \cap \E = \sourcesymbol{\properfiltertwo} \cap \E = \Efilterofproperfilter{\properfiltertwo}$, we have
\[
\pff(\filtercompose{\properfilterone}{\properfiltertwo}) = \germ{st}{\Efilterofproperfilter{\filtercompose{\properfilterone}{\properfiltertwo}}} = \germ{st}{\Efilterofproperfilter{\properfiltertwo}} = \germ{s}{\Efilterofproperfilter{\properfilterone}}\germ{t}{\Efilterofproperfilter{\properfiltertwo}} = \pff(\properfilterone)\pff(\properfiltertwo).
\]

To see that $\pff$ is continuous, fix $s \in S$ and take any open set $\openone \subseteq \EfiltersIB{\sinv{s}s}$. Then $\propergermsbasic{s}{\openone}$ is an arbitrary basic open set in $\groupoidofpropergerms$. A routine argument shows that \[
\inv{\pff}(\propergermsbasic{s}{\openone}) = \properfiltersIB{s} \cap \inversesourcesymbol{\iofinv{\openone}}.
\]
Since $\sourcesymbolname$ is continuous and \cref{prop:gopfunits_homeomorphic_to_setofefilters_patch} implies that $\iofinv{\openone}$ is open in $\idempotentfiltersof{\setofproperfilters}$, we deduce that $\inversesourcesymbol{\iofinv{\openone}}$ is open in $\groupoidofproperfilters$. Since $\properfiltersIB{s}$ is also open, $\inv{\pff}(\propergermsbasic{s}{\openone}) = \properfiltersIB{s} \cap \inversesourcesymbol{\iofinv{\openone}}$ is open in $\groupoidofproperfilters$.

Finally, to see that $\pff\colon \groupoidofproperfilters \to \groupoidofpropergerms$ is open, fix $s \in S$ and let $T \finsubseteq \down{s}$, so that $\properfiltersPB{s}{T}$ is a basic open set in $\groupoidofproperfilters$. Since both $\textbf{d}$ and $\epsilon$ are open maps by \cref{cor:gopf_etale_Hausdorffunits_patch,prop:gopfunits_homeomorphic_to_setofefilters_patch} respectively, we deduce that $\pff(\properfiltersPB{s}{T}) = \propergermsbasic{s}{\iof{\sourcesymbol{\properfiltersPB{s}{T}}}}$ is a basic open set in $\groupoidofpropergerms$, and hence $\pff$ is open.
\end{proof}

The map $\pff\colon \groupoidofproperfilters \to \groupoidofpropergerms$ restricts nicely to the ultrafilters.

\begin{lemma} \label{lem:uff_well_defined_lemma}
We have $\pff(\groupoidofultrafilters) = \groupoidofultragerms$.
\end{lemma}

\begin{proof}
Fix $\ultrafilterone \in \groupoidofultrafilters$ and let $s \in \ultrafilterone$. Then $\pff(\ultrafilterone) = \germ{s}{\Efilterofproperfilter{\ultrafilterone}}$. By \cref{lem:pff_wd_rp}\cref{item:pff_notation}, we have $\Efilterofproperfilter{\ultrafilterone} \in \setofEultrafilters$, and hence $\pff(\ultrafilterone) \in \groupoidofultragerms$. For the reverse containment, fix $\germ{s}{\Efilterone} \in \groupoidofultragerms \subseteq \groupoidofpropergerms$. Since $\pff$ maps $\groupoidofproperfilters$ onto $\groupoidofpropergerms$, there is some $\ultrafilterone \in \groupoidofproperfilters$ such that $\pff(\ultrafilterone) = \germ{s}{\Efilterone}$. Fix $t \in \ultrafilterone$. Then $\germ{t}{\Efilterofproperfilter{\ultrafilterone}} = \pff(\ultrafilterone) = \germ{s}{\Efilterone}$. In particular, $\Efilterone = \Efilterofproperfilter{\ultrafilterone}$, and hence $\sourcesymbol{\ultrafilterone} \cap \E = \Efilterofproperfilter{\ultrafilterone} = \Efilterone \in \setofEultrafilters$. Therefore, \cref{prop:gopfunits_homeomorphic_to_setofefilters_patch} implies that $\sourcesymbol{\ultrafilterone} \in \units{\groupoidofultrafilters}$, and so $\ultrafilterone = \filtercompose{\ultrafilterone}{\sourcesymbol{\ultrafilterone}} \in \groupoidofultrafilters$, because $\setofultrafilters$ is an ideal in $\setofproperfilters$ by \cite[Proposition~2.41]{Bic21}. Thus, $\germ{s}{\Efilterone} = \pff(\ultrafilterone) \in \pff(\groupoidofultrafilters)$.
\end{proof}

\cref{cor:uff} is immediate from \cref{lem:uff_well_defined_lemma,thm:pff}, since the patch topologies on $\groupoidofultrafilters$ and $\groupoidofultragerms$ are the subspace topologies relative to the patch topologies
on $\groupoidofproperfilters$ and $\groupoidofpropergerms$, respectively.

\begin{corollary} \label{cor:uff}
The restriction $\pff|_{\groupoidofultrafilters}\colon \groupoidofultrafilters \to \groupoidofultragerms$ is a topological groupoid isomorphism with respect to the patch topologies.
\end{corollary}

For all $s \in S$, define $\ultrafiltersIB{s} \coloneqq \properfiltersIB{s}\cap\setofultrafilters$. We generalise \cite[Proposition~5.18]{LL13}.

\begin{proposition} \label{prop:ultrafilters_inclusive}
The collection $\ultrafiltersIBcollection$ is a basis for the patch topology on $\groupoidofultrafilters$.
\end{proposition}

\begin{proof}
By \cref{cor:uff}, it suffices to show that $(\pff(\ultrafiltersIB{s}))_{s \in S}$ is a basis for the topology on $\groupoidofultragerms$. For each $s \in S$ and open set $A \subseteq \EultrafiltersIB{\sinv{s}s}$, define $\ultragermsbasic{s}{A} \coloneqq \{ \germ{s}{\Efilterone} : \Efilterone \in A \}$. The collection of all such sets $\ultragermsbasic{s}{A}$ is a basis for the topology on $\groupoidofultragerms$. Observe that
\[
\pff(\ultrafiltersIB{s}) = \{\germ{s}{\Efilterofproperfilter{\ultrafilterone}} : \ultrafilterone \in \ultrafiltersIB{s} \} = \{\germ{s}{\Efilterone} : \Efilterone \in \EultrafiltersIB{\sinv{s}s}\} = \ultragermsbasic{s}{\EultrafiltersIB{\sinv{s}s}},
\]
and so $(\pff(\ultrafiltersIB{s}))_{s \in S} = (\ultragermsbasic{s}{\EultrafiltersIB{\sinv{s}s}})_{s \in S}$. Thus it suffices to show that $(\ultragermsbasic{s}{\EultrafiltersIB{\sinv{s}s}})_{s \in S}$ is a basis for the topology on $\groupoidofultragerms$. Fix $s \in S$ and take any open set $A \subseteq \EultrafiltersIB{\sinv{s}s}$. Let $\germ{s}{\Efilterone} \in \ultragermsbasic{s}{A}$. Since the collection of sets of the form $\ultragermsbasic{s}{A}$ is a basis for the topology on $\groupoidofultragerms$, it suffices to find $t \in S$ such that $\germ{s}{\Efilterone} \in \ultragermsbasic{t}{\EultrafiltersIB{\sinv{t}t}} \subseteq \ultragermsbasic{s}{A}$. It follows from \cite[Lemma~2.26]{Law12} (which applies to the collection of idempotents in any inverse semigroup) that $A = \bigcup_{e \in X} \EultrafiltersIB{e}$ for some $X \subseteq \E$. Thus, there is some $e \in X$ such that $\Efilterone \in \EultrafiltersIB{e}$, and so $\germ{s}{\Efilterone} \in \ultragermsbasic{s}{\EultrafiltersIB{e}} \subseteq \ultragermsbasic{s}{A}$. By putting $t \coloneqq se$ and observing that $\sinv{t}t =~ \sinv{s}se \in \Efilterone$ and $\sinv{t}t \npo e$, we see that
\[
\germ{s}{\Efilterone} \in \ultragermsbasic{t}{\EultrafiltersIB{\sinv{t}t}} \subseteq \ultragermsbasic{s}{\EultrafiltersIB{e}} \subseteq \ultragermsbasic{s}{A}. \qedhere
\]
\end{proof}

\begin{remark}
Knowing that $\ultrafiltersIBcollection$ is a basis for the patch topology on $\groupoidofultrafilters$ leads to the following characterisation of convergence of nets in $\groupoidofultrafilters$: for any net $\indexbyin{\ultrafilterone}{k}{K} \subseteq \groupoidofultrafilters$ and any $\ultrafilterone \in \groupoidofultrafilters$, $\indexbyin{\ultrafilterone}{k}{K}$ converges to $\ultrafilterone$ if and only if, for each $u \in \ultrafilterone$, there exists $k_u \in K$ such that $u \in \ultrafilterone_k$ for all $k ~\succeq~ k_u$.
\end{remark}

\begin{remark}
Define $\groupoidoftightfilters \coloneqq \inv{\pff}(\groupoidoftightgerms)$, where $\groupoidoftightgerms$ is Exel's tight groupoid given in \cite[Theorem~13.3]{Exe08}. Since $\pff\colon \groupoidofproperfilters \to \groupoidofpropergerms$ is a topological groupoid isomorphism and $\groupoidoftightgerms$ is a subgroupoid of $\groupoidofpropergerms$, $\groupoidoftightfilters$ is a topological subgroupoid of $\groupoidofproperfilters$, which is topologically isomorphic to $\groupoidoftightgerms$. By \cite[Lemma~5.9]{LL13}, $\groupoidoftightfilters$ is the groupoid denoted by $G_t(S)$ in \cite{LL13}. Furthermore, $\groupoidoftightfilters$ is the reduction of $\groupoidoffilters$ with respect to a certain coverage notion referred to in \cite[Corollary~6.8]{Cas20}.
\end{remark}

\begin{remark}
Let $\setofEfilters$ be the set of filters in the inverse semigroup $\E$, which we can identify with the \emph{spectrum} of $\E$ defined in \cite[Definition~10.1]{Exe08}. There is a groupoid $\groupoidofgerms$ of germs associated to $(S,\setofEfilters)$, as $\groupoidofpropergerms$ is to $(S,\setofEproperfilters)$. The groupoid $\groupoidofgerms$ of germs can be identified with Paterson's universal groupoid $\Patersonsgroupoid$ \cite[Definition~5.14]{Ste10}, and $\pff$ extends to a topological isomorphism $\lff$ from $\groupoidoffilters$ to $\groupoidofgerms$ \cite[p5]{Cas20}. Moreover, the map $s \mapsto \up{s}$ from $S$ to $\setoffilters$ is a faithful homomorphism of inverse semigroups. Therefore, the diagram from \cref{diag:roadmap} extends as follows:

\begin{center}
\begin{tikzcd}[column sep=tiny]
S\arrow[rr, hook, "\up{}"] & & \groupoidoffilters\arrow[d, hook, two heads, "\lambda"] & \ge & \groupoidofproperfilters \arrow[d, hook, two heads, "\pff"] & \ge & \groupoidoftightfilters \arrow[d, hook, two heads, "\pff|_{\groupoidoftightfilters}"] & \ge & \groupoidofultrafilters \arrow[d, hook, two heads, "\pff|_{\groupoidofultrafilters}"] \\
\Patersonsgroupoid & \cong & \groupoidofgerms & \ge & \groupoidofpropergerms & \ge & \groupoidoftightgerms & \ge & \groupoidofultragerms
\end{tikzcd}
\end{center}
\end{remark}

\begin{remark}
The set $\E$ is a \emph{semilattice} in the sense that $ef$ is the greatest lower bound of $\{e,f\}$, for all $e,f \in \E$. The semilattice $E$ is \emph{compactable} in the sense of \cite{Law10cs} if and only if $\setoftightEfilters = \setofEultrafilters$, by \cite[Theorem~2.5]{Law10cs}. When $\setoftightEfilters = \setofEultrafilters$, $\groupoidoftightgerms$ coincides with $\groupoidofultragerms$. Therefore, $\E$ is compactable if and only if the map $\pff|_{\groupoidofultrafilters}$ is a topological groupoid isomorphism from the groupoid $\groupoidofultrafilters$ of ultrafilters to Exel's tight groupoid $\groupoidoftightgerms$.
\end{remark}

\end{document}